\documentclass[tikz,border=10pt]{article}
\UseRawInputEncoding
\usepackage[english]{babel}
\usepackage[utf8x]{inputenc}
\usepackage[T1]{fontenc}

\usepackage{tikz}
\usepackage{pgflibraryarrows}
\usepackage{pgflibrarysnakes}

\usetikzlibrary{decorations.markings}

\usepackage[a4paper,top=3cm,bottom=2cm,left=3cm,right=3cm,marginparwidth=1.75cm]{geometry}

\usepackage{amsmath}
\usepackage{graphicx}
\usepackage[colorinlistoftodos]{todonotes}
\usepackage[colorlinks=true, allcolors=blue]{hyperref} \usepackage{mathrsfs}
\usepackage{amssymb}
\usepackage{lineno}
\usepackage{bbding}
\usepackage{fancyhdr}\usepackage{graphics}
\usepackage{titlesec}\usepackage{titletoc}
\usepackage{indentfirst}
\usepackage{amsmath}\usepackage{amsfonts}\usepackage{latexsym}
\usepackage{eucal}\usepackage{epsfig}\usepackage{mathrsfs,amsmath,amsthm}
\usepackage{geometry}

\newtheorem{theorem}{Theorem}[section]
 
 \newtheorem{proposition}[theorem]{Proposition}
 \newtheorem{corollary}[theorem]{Corollary}
 \newtheorem{definition}[theorem]{Definition}

 \newcount\refno
 \def\RR{{\mathbb R}}
 \def\NN{{\mathbb N}}
 \def\SS{{\mathbb S}}
 
 \def\ZZ{{\mathbb Z}}

 \def\SF{{\mathscr F}}
 
 \title{\bf The Smoothness of  kernel in  Hardy spaces
 \thanks{1} \footnote{E-mail: huzhuoran010@163.com[ZhuoRan Hu].}}

\author{{ZhuoRan Hu}\\
{\small }\\
{\small Beijing 100048, China}}

\begin{document}
\maketitle \setcounter{page}{1} \pagestyle{myheadings}
 \markboth{Hu}{The Smoothness of  kernel in  Hardy spaces}
\begin{abstract}
This paper provides a study of  problems related to Hardy spaces left by G.\,Weiss in \cite{We}.
First, We will prove that the Hardy spaces $H^p(\RR^n)$ can be characterized by a fixed Lipschitz function.

\vskip .2in
 \noindent
 {\bf 2000 MS Classification:}
 \end{abstract}

\setcounter{page}{1}

\section{Introduction}
Fefferman and Stein in \cite{CF} showed that $H^p(\RR^n)$ can be defined as follows:
\begin{theorem}\label{abc7}
For $0<p\leq\infty$, let $f$ be a distribution, then the following conditions are equivalent:

(1)There is a $\phi\in S(\RR^n)$ with $\int \phi(x) dx\neq0$ so that $M_\phi f\in L^p(\RR^n)$

(2)The distribution $f$ is a bounded distribution and $\sup_{|u-x|<t}(f*P_t)(u)\in L^p(\RR^n)$.

(3)$M_Ff(x)=\sup_{\phi\in S_{F}}\sup_{t>0}\left(f\ast\phi_t\right)(x)\in L^p(\RR^n)$, where $F=\{\|\cdot\|_{a, b}\}$ is any finite collection of seminorms on $S(\RR^n)$, and $S_{F}$ is the subset of $S(\RR^n)$ controlled by
this collection of seminorms:$$S_{F}=\left\{\phi\in S(\RR^n): \|\phi\|_{a, b}\leq1\  \hbox{for}\  \hbox{any}\  \|\cdot\|_{a, b}\in F \right\}.$$
\end{theorem}
They also discussed the minimal conditions on $\phi$ so that  $M_\phi f\in L^p(\RR^n)$ with $\|M_\phi f\|_{L^p(\RR^n)}\lesssim\|f\|_{H^p(\RR^n)}$  whenever $f\in H^p$ $(p\leq1)$.

(a)For $\phi$ that have compact support, it suffices to have $\phi\in \Lambda^{\gamma} $ for some $\gamma>n(p^{-1}-1)$.

(b)For $\phi$ not having compact support but vanishing at infinity and satisfying $\left|\partial_x^\gamma \phi(x)\right|\lesssim (1+|x|)^{-N}$ for  $|\gamma|=[n(p^{-1}-1)]+1$, it suffices to have $Np>n$.

From \cite{DL} and  \cite{TB}, the Lipchitz spaces $ \Lambda^{\gamma}$ can be paired with $H^p$ if $\gamma=[n(p^{-1}-1)] $ $(0<p<1)$. That is for any $f\in H^p $, the following holds:
\begin{eqnarray}\label{abc5}
\|f\|_{H^p}= \sup_{\|g\|_{\Lambda^{\gamma}}\leq1}\left|\int f(x)g(x)dx\right|.
\end{eqnarray}
 From \cite{MS2},  $H^p(\RR^n)$ ($\frac{1}{1+\gamma}<p\leq1$) spaces can also   be defined as:
\begin{eqnarray}\label{abc6}
\|f\|^p_{H^p(\RR^n)}=\int_{\RR^n}\left|f_{\gamma}^*(x)\right|^pdx,
\end{eqnarray}
where $f\in \left(\Lambda^{\gamma}\right)'$, and $f_{\gamma}^*(x)$ is defined by\,(\ref{ssr3}).

From (\ref{abc5}), (\ref{abc6})  and Theorem\,\ref{abc7}, we could see that the smoothness of the kernel in Theorem\,\ref{abc7} may be reduced. Thus the problem reducing smoothness of $\phi$ in Hardy spaces was proposed by some mathematicians. The example that $\phi=\chi_{[-\frac{1}{2}, \frac{1}{2}]}(x)$ shows that the assumption of smoothness of $\phi$ can not be removed in the definition of $H^p$. This shows that the Hardy-Littlewood maximal function $Mf$ can not characterize any $H^p$ for $0<p\leq1$. Thus we wish to replace the Schwartz function $\phi$ in Theorem\,\ref{abc7}   by  a fixed Lipschitz function.

 Our  \textbf{first  result} is Theorem\,\ref{pp3}. We will prove the following for $\frac{1}{1+\gamma}<p\leq1$:
\begin{eqnarray}\label{abc8}
\|f\|_{H^p(\RR^n)}\sim_{p, \gamma, \phi}\|(f*\phi)_{\nabla}\|_{L^p(\RR^n)}\sim_{ p,  \gamma, \phi} \|(f*\phi)_{+}\|_{L^p(\RR^n)},
\end{eqnarray}
where  $\phi$ is a fixed  Lipschitz function with compact support satisfying $\int \phi(x)dx\sim1$.  $(f*\phi)_{\nabla}(x)$ and $(f*\phi)_{+}(x)$ are    non-tangential maximal function and radial maximal function  defined as:
\begin{eqnarray*}
(f*\phi)_{\nabla}(x)=  \sup_{ |x-u|< t}  \left|f*\phi_t(u)\right|,\  (f*\phi)_{+}(x)=  \sup_{ t>0} \left| f*\phi_t(x)\right|,\ \hbox{where}\ \  \displaystyle{\phi_t(x)=t^{-n}\phi\left(xt^{-1}\right)}.
\end{eqnarray*}  From the above (\ref{abc8}), we could also see that the norm of  $(f*\phi)_{+}$ and the norm of  $(f*\phi)_{\nabla}$ are equivalent when   $\phi\in \Lambda^{\gamma}$  with a compact support.  (\ref{abc8}) is different to (\ref{abc6}) that  $\phi$ in (\ref{abc6}) is not  fixed.

 In 1983, Han in \cite{H8} gave another characterization of  $H^1_\phi(\RR)$ with Carleson measure.
Han also proved that there exists a $\phi$ which is a  Lipschitz function
satisfying the Formulas\,(\ref{lp3}, \ref{lp4}) so that  $(f*\phi)_{\nabla}\in L^1(\RR)\Rightarrow f(x)=0\,a.e.x\in\RR$.
\begin{eqnarray}\label{lp3}
\left|\phi(x)\right|\lesssim \frac{1}{(1+|x|)^{1+\gamma}},
\end{eqnarray}
\begin{eqnarray}\label{lp4}
\left|\phi(x+h)-\phi(x)\right|\lesssim \frac{|h|^{\gamma}}{(1+|x|)^{1+2\gamma}},\ \ \hbox{if}\ |h|\lesssim|x|/2.
\end{eqnarray}
when $0<\gamma\leq1$. However, when we replace the Formulas\,(\ref{lp3}, \ref{lp4}) with Formulas\,(\ref{lp1}, \ref{lp2}), we could deduce the Proposition\,\ref{lp51}: $f\in H^p(\RR^n)\Rightarrow (f*\phi)_{\nabla}\in L^p(\RR^n)$ for $\frac{1}{1+\gamma}<p\leq1$.

Our  \textbf{second  result} is Theorem\,\ref{upoo}, that  the following  holds for $\frac{1}{1+\gamma}<p\leq1$:
\begin{eqnarray}\label{abc88}
\|f\|_{H^p(\RR^n)}\sim_{p, \gamma, \phi}\|(f*\phi)_{\nabla}\|_{L^p(\RR^n)}\sim_{ p,  \gamma, \phi} \|(f*\phi)_{+}\|_{L^p(\RR^n)},
\end{eqnarray}
where  $\phi$ is a fixed  Lipschitz function without compact support satisfying Formulas\,(\ref{up1}, \ref{up2}) and $\int \phi(x)dx\sim1$.

\textbf{Notation}: Let $S(\RR^n)$ be the space of  $C^{\infty}$ functions on $\RR^n$ with the Euclidean distance rapidly decreasing together with their derivatives\,(Schwartz Class), $S'(\RR^n)$ the tempered distributions. In the following of this paper, we assume that $0<\gamma\leq1$,
$\alpha\in\NN^n$ satisfying:
$$\alpha=(\alpha_i)_{i=1}^{n},\ \ \ \hbox{where}\, \alpha_i\in\NN,\ \ \hbox{and}\,|\alpha|=\sum_{i=1}^n\alpha_i.$$
We use $\SS^n$ to denote the unit sphere in $\RR^{n+1}$, $B(x, r)$  to denote  the set: $B(x, r)=\{y: |x-y|<r\}.$ $B(x, r_1)\backslash B(y, r_2)$ is denoted as the set: $B(x, r_1)\bigcap B(y, r_2)^c.$

\section{ Lipschitz function with compact support in $\RR^n$}

\begin{definition}[The Lipschitz function]
For $\phi\in C(\RR^n)$, $n\in \NN$,
$H^{\gamma}(\phi)$ is denoted as:
$$H^{\gamma}(\phi)=\sup_{x, y\in\RR^n, x\neq y}|\phi(x)-\phi(y)|/|x-y|^{\gamma};$$
The Lipschitz function $\Lambda^{\gamma}$ is defined as $\Lambda^{\gamma}=\{f: \sup_{x\in\RR^n}|f(x-y)-f(x)|\leq C |y|^{\gamma}\},$
and $\left(\Lambda^{\gamma}\right)'$ is  denoted as the dual space of $\Lambda^{\gamma}$.
\end{definition}
For $f\in \left(\Lambda^{\gamma}\right)'$, the maximal function  $f_{\gamma}^*(x)$ in $\RR^n$ is defined as:
\begin{eqnarray}\label{ssr3}
f_{\gamma}^*(x)=\sup_{\phi,r}\left\{\left|\int_{\RR^n} f(y)\phi(y)dy\right|/r^n:r>0,{\rm
supp}\,\phi\subset B(x, r), H^\gamma(\phi) \leq r^{-\gamma}, \|\phi\|_{L^{\infty}}\leq1 \right\}.
\end{eqnarray}

For  $f\in S'(\RR^n)$, $f_{S\gamma}^*(x)$ is defined as:
\begin{eqnarray*}f_{S\gamma}^*(x)=\sup_{\phi,r}\left\{\left|\int_{\RR^n} f(y)\phi(y)dy\right|/r^n:r>0,{\rm
supp}\,\phi\subset B(x, r), H^\gamma(\phi) \leq r^{-\gamma},  \right. \\ \nonumber \phi\in S(\RR^n), \|\phi\|_{L^{\infty}}\leq1 \bigg\}.
\end{eqnarray*}
or
\begin{eqnarray}\label{mm1}
f_{S\gamma}^*(x)=\sup_{\psi,r>0}\left\{\left|\int_{\RR^n} f(y)\psi\left(\frac{x-y}{r} \right)dy\right|/r^n:  \psi(t)\in S(\RR^n), \right.\\
\nonumber  supp\,\psi(t)\subset B(0, 1), \|\psi\|_{L^{\infty}}\leq1, H^{\gamma}\psi\leq1  \bigg\}.
\end{eqnarray}

\begin{definition}[$\mathbf{M_{\phi}f(x) } $ ]\label{mm2} For $f\in S'(\RR^n)$, $M_{\phi}f(x)$ is defined as
\begin{eqnarray*}
 M_{\phi}f(x)  =\sup_{r>0}\left\{\left|\int_{\RR^n} f(y)\phi\left(\frac{x-y}{r} \right)dy\right|/r^n:  \phi(t)\in S(\RR^n) \right\}.
\end{eqnarray*}
 \end{definition}

\begin{proposition}\cite{Stein} \label{tan} For fixed numbers $0<b\leq a$, F(x, r) is  a function defined on $\RR_+^{n+1}$, its nontangential maximal function $F_a^*(x)$ is defined as
 $F_a^*(x)=\sup_{|x- y|<ar}|F(y, r)|.$ If $F_a^*(x)\in L^1(\RR^n)$ or $F_b^*(x)\in L^1(\RR^n)$, then we could obtain  the following inequality for $p>0$:
 \begin{eqnarray}\label{no1}
 \int_{\RR^n} |F_a^*(x)|^pdx\leq c \left(\frac{a+b}{b}\right)^n \int_{\RR^n} |F_b^*(x)|^pdx ,
 \end{eqnarray}
where c is a constant independent on $F$, a, b.
\end{proposition}
\begin{proposition}\label{sk}
For $f\in L^1(\RR^n)$, $0<p<\infty$ we could obtain
$  f_{S\gamma}^*(x)= f_{\gamma}^*(x)\ \ \ a.e.x\in\RR^n .$
Further more, if
$ \int_{\RR^n} |f_{\gamma}^*(x)|^p dx\leq \infty $ or $ \int_{\RR^n} |f_{S\gamma}^*(x)|^p dx\leq \infty $,
then the following holds
 $$\int_{\RR^n} |f_{\gamma}^*(x)|^p dx \sim \int_{\RR^n} |f_{S\gamma}^*(x)|^p dx.$$
\end{proposition}
\begin{proof}
We will prove the following (\ref{exuu1})\,first:
\begin{eqnarray}\label{exuu1}
 f_{S\gamma}^*(x)= f_{\gamma}^*(x)\ \ \ a.e.x\in\RR^n.
\end{eqnarray}
 It is easy to see that  $f_{S\gamma}^*(x)\leq f_{\gamma}^*(x)$. If $\phi$ satisfies $ H^\gamma(\phi) \leq r^{-\gamma}$ and ${\rm
supp}\,\phi\subset B(x, r)$, then $\phi$ is a  continuous function  with compact support. Thus there exists sequence $ \{\psi_k\}_k \subset S(\RR^n)$ with $\lim_{k\rightarrow\infty}\|\psi_k-\phi\|_\infty = 0$, $\|\psi_k -\phi\|_\infty\neq0$.
Denote $\delta_k(x)$ as
$\delta_k(x)=\left|\int_{B(x, r)} f(y)\left(\phi(y)-\psi_k(y)\right)dy/r^n\right|,$
then
$\delta_k(x)\leq M f(x)\|\psi_k -\phi\|_\infty.$
Let $i_k$ to be $i_k= \|\psi_k -\phi\|_\infty$,
then we could obtain:
$$\left\{x: \delta_k(x)>\alpha\right\}\subseteq \left\{x: M f(x) > \frac{\alpha}{i_k}\right\}.$$
By the fact that $M$ is weak-(1, 1) bounded,  the following could be obtained from the above inequality for any $\alpha>0$:
$$\left|\left\{x: \delta_k(x)>\alpha\right\}\right|\leq \frac{1}{\alpha} \|f\|_{L^1(\RR^n)}\|\psi_k-\phi\|_\infty.$$
Thus by $\lim_{k\rightarrow\infty}\|\psi_k-\phi\|_\infty = 0$, we could have
 $$\lim_{k\rightarrow+\infty} \left|\left\{x: \delta_k(x)>\alpha\right\}\right|=0 .$$
Then there exists a sequence $\{k_j\}\subseteq \{k\}$ such that $$\int_{\RR^n} f(y)\phi(y)dy/r^n=\lim_{k_j\rightarrow\infty}\int_{\RR^n} f(y)\psi_{k_j}(y)dy/r^n,\ \ \ a.e.x\in\RR^n $$ for $f\in L^1(\RR^n)$.
Thus we could obtain:$$\int_{\RR^n} f(y)\phi(y)dy/r^n\leq  f_{S\gamma}^*(x)\ \ \ a.e.x\in\RR^n $$
for any $\phi$ satisfies $H^\gamma(\phi) \leq r^{-\gamma}$ and ${\rm
supp}\,\phi\subset B(x, r)$.
We could then deduce $$\sup_{\phi,r>0}\left|\int_{\RR^n} f(y)\phi(y)dy/r^n\right|\leq  f_{S\gamma}^*(x)\ \ \ a.e.x\in\RR^n .$$
Thus with the fact that $f_{S\gamma}^*(x)\leq f_{\gamma}^*(x)$, we could have
$$  f_{S\gamma}^*(x)= f_{\gamma}^*(x)\ \ \ a.e.x\in\RR^n.$$

Then we will prove the following\,(\ref{max}):
\begin{eqnarray}\label{max}
 \int_{\RR^n} |f_{\gamma}^*(x)|^p dx \sim \int_{\RR^n} |f_{S\gamma}^*(x)|^p dx
\end{eqnarray}
Let $E$ denote a set defined as $E=\left\{x: f_{S\gamma}^*(x)= f_{\gamma}^*(x)\right\}$. Next we will prove that for any $x_0\in\RR^n$, there is a point $\overline{x}_0 \in E$ such that
\begin{eqnarray}\label{exuu2}
f_{S\gamma}^*(x_0)\lesssim f_{S\gamma}^*(\overline{x}_0).
\end{eqnarray}
Notice that for $x_0\in\RR^n$, there exist $r_0>0$ and $\phi_0$ satisfying: $supp\,\phi_0\subset B(x_0, r_0)$, $\phi_0\in S(\RR^n)$, $H^\gamma(\phi_0) \leq r^{-\gamma}$, $\|\phi_0\|_{L^{\infty}}\leq1$, such that the following inequality holds:
$$\left|\frac{1}{r_0^n}  \int f(y)\phi_0 (y)dy\right|\geq \frac{1}{2} f_{S \gamma}^*(x_0).$$
Notice that $|\RR^n \backslash E|=| E^c|=0$ leads to the fact that  $E$ is dense in $\RR^n$, thus there exists a $\overline{x}_0\in E$ with $d (x_0, \overline{x}_0)\leq\frac{r_0}{4}$.
Then $supp\,\phi_0\subset B(\overline{x}_0, 4r_0)$ holds, and we could obtain the following
$$\left|\frac{1}{r_0^n}  \int f(y)\phi_0 (y)dy\right|\leq C f_{S\gamma}^*(\overline{x}_0),$$
where C is a constant independent on $f$, $\gamma$ and $r_0$. Thus  (\ref{exuu2})\,could be obtained. By (\ref{exuu2}), we could deduce the following:
\begin{eqnarray}\label{uu3}
 \int_E |f_{S\gamma}^*(x)|^p dx<\infty\, \Rightarrow\,  \int_{\RR^n} |f_{S\gamma}^*(x)|^p dx\sim \int_E |f_{S\gamma}^*(x)|^p dx.
\end{eqnarray}
In the same way, we could conclude that
\begin{eqnarray}\label{uu4}
\int_{\RR^n} |f_{\gamma}^*(x)|^p dx \sim \int_E |f_{\gamma}^*(x)|^p dx.
\end{eqnarray}
From Formula\,(\ref{exuu1})\,we could deduce:
\begin{eqnarray}\label{uu5}
\int_E |f_{\gamma}^*(x)|^p dx = \int_E |f_{S\gamma}^*(x)|^p dx .
\end{eqnarray}
The above Formula\,(\ref{uu5}) together with Formulas\,(\ref{uu3}, \ref{uu4})\,lead to\,(\ref{max}) if
$ \int_{\RR^n} |f_{\gamma}^*(x)|^p dx< \infty $ or $ \int_{\RR^n} |f_{S\gamma}^*(x)|^p dx< \infty $.
This proves the proposition.
\end{proof}

\begin{proposition}\label{pp1}
For $\phi(x)\in \Lambda^{\gamma}$, $supp\,\phi(x)\subseteq\left\{x\in\RR^n: |x|<1\right\}$, $\frac{1}{1+\gamma}<p\leq 1$, $|\phi(x)| \leq1$  $\int\phi(x)dx=1$,  we could deduce the following  for any $f\in H^p(\RR^n)$:
\begin{eqnarray}\label{tan13}
\|f\|_{H^p(\RR^n)}\sim_{\gamma, p, \phi}\|(f*\phi)_{\nabla}\|_{L^p(\RR^n)}.
\end{eqnarray}
\end{proposition}

\begin{proof}

Let  $f\in L^1(\RR^n)$,  $\psi\in S(\RR^n)$ with $\int\psi(x)dx\sim1$. There exists sequence $\{\phi^m(x): \phi^m(x) \in S(\RR^n)\}_{m\in\NN}$ satisfying:
$$\left\|\phi^m(x)- \phi(x)\right\|_{L^\infty(\RR^n)}\lesssim\frac{1}{m},\ \ \ \ \int\phi^m(x)dx\sim1.$$
 $\varphi\in S(\RR^n)$ is a fixed function so that
\begin{eqnarray*}
\left\{ \begin{array}{cc}
\varphi(\xi)=0 \ \ \hbox{for}\,|\xi|\geq1 \\
\\
\varphi(\xi)=1\ \ \hbox{for}\,|\xi|\leq1/2.
\end{array}\right.
\end{eqnarray*}
We use  $\varphi^k\in S(\RR^n)$ to denote as:
\begin{eqnarray*}
\left\{ \begin{array}{cc}
\varphi^k(\xi)=\varphi(\xi) \ \ \hbox{for}\,k=0, \\
\\
\varphi^k(\xi)=\varphi(2^{-k}\xi)-\varphi(2^{1-k}\xi)\ \ \hbox{for}\,k\geq1.
\end{array}\right.
\end{eqnarray*}
Then we could have that
$$1=\sum_{k=0}^{\infty}\varphi^k(\xi).$$
Notice that $\int\phi^m(x)dx\sim1$, thus $(\SF\phi^m)(2^{-k_0}\xi)\geq C$  for $|\xi|\leq1$, where $C$ and $k_0$ are independent on $m$. Let $\eta_m^{k}$ to be
$$(\SF\eta_m^{k})(\xi)=\frac{\varphi^k(\xi)(\SF\psi)(\xi)}{(\SF\phi^m)(2^{-k-k_0}\xi)},$$
where $\SF$ denotes the Fourier transform.
Then we could obtain that:
\begin{eqnarray*}
(\SF\psi)(\xi)&=&\sum_{k=0}^{\infty}\frac{\varphi^k(\xi)(\SF\psi)(\xi)}{(\SF\phi^m)(2^{-k-k_0}\xi)}(\SF\phi^m)(2^{-k-k_0}\xi)\\ \nonumber &=&\sum_{k=0}^{\infty}(\SF\eta_m^{k})(\xi)(\SF\phi^m)(2^{-k-k_0}\xi).
\end{eqnarray*}

Thus
\begin{eqnarray}\label{tan17}
\psi(x)=\sum_{k=0}^{\infty}\eta_m^{k}\ast\phi^m_{2^{-k-k_0}}(x).
\end{eqnarray}
By the fact that $\sup_{\xi\in\RR^n}|\partial_{\xi}^{\alpha'}(\SF\phi^m)(\xi)|\leq C_{ \alpha'}$ and \begin{eqnarray}\label{tan14}
\sup_{\xi\in\RR^n}\left||\xi|^{\alpha}\partial_{\xi}^{\alpha'}(\varphi^k(\xi)(\SF\psi)(\xi))\right|\lesssim_{\alpha, \alpha', M} 2^{-kM}\ \hbox{for\ any\ }M>0,
\end{eqnarray} where $C_{ \alpha'}$ is a
constant independent on $m$, we could deduce that
\begin{eqnarray}\label{tan15}
\sup_{\xi\in\RR^n}\left||\xi|^{\alpha}\partial_{\xi}^{\alpha'}(\SF\eta_m^{k})(\xi)\right|\leq C_{\alpha, \alpha', M, k_0}2^{-kM}\ \hbox{for\ any \ }M>0,
\end{eqnarray}
where $C_{\alpha, \alpha', M, k_0}$ is a
constant independent on $m$ and $k$. Thus we could have:
\begin{eqnarray}\label{tan16}
\left|\int_{\RR^n}\eta_m^{k}\left(u\right)\left(1+2^{k+k_0}|u|\right)^{ N}du \right| \leq C_{k_0, N}2^{-k},
\end{eqnarray}
where $C_{k_0, N}$ is a constant independent on $m$.
Then by Formulas\,(\ref{tan17}) with the fact that  $f\in  L^1(\RR^n)$ we have
\begin{eqnarray}\label{tan18}
M_{\psi}f(x)&=& \sup_{r>0}\left|\int_{\RR^n} f(y)\frac{1}{r^n}\psi\left(\frac{x-y}{r} \right)dy\right|
\\&=&\nonumber C\sup_{r>0}\sum_{k=0}^{+\infty}\left|\int_{\RR^n}\int_{\RR^n} f(y)\eta_m^{k}\left(\frac{s}{r}\right)\frac{1}{r^n}\phi^m\left(\frac{x-y-s}{2^{-k-k_0}r}\right)\frac{ds}{(2^{-k-k_0}r)^{n}}dy\right|
\\&\leq&\nonumber C \sum_{k=0}^{+\infty} \left|\int_{\RR^n}\eta_m^{k}\left(\frac{s}{r}\right)\left(1+\frac{|s|}{2^{-k-k_0}r} \right)^{N}\frac{ds}{r^n}\right| \sup_{ r>0, s\in\RR^n}\left|\int_{\RR^n} f(y)\phi^m\left(\frac{x-y-s}{r} \right)\left(1+\frac{|s|}{r} \right)^{-N}\frac{dy}{r^n}\right|,
\end{eqnarray}
where $C$ is a constant independent on $m$. From Formula\,(\ref{tan16}) and Formula\,(\ref{tan18}) we could obtain:
\begin{eqnarray}\label{tan27}
M_{\psi}f(x)&\lesssim&  \sup_{ r>0, s\in\RR^n}\left|\int_{\RR^n} f(y)\phi^m\left(\frac{x-y-s}{r} \right)\left(1+\frac{|s|}{r} \right)^{-N}\frac{dy}{r^n}\right|
\\&\lesssim&\nonumber   \left(\sup_{0\leq |s|< r}+\sum_{k=1}^{\infty}\sup_{ 2^{k-1}r\leq |s|< 2^{k}r}\right)\left|\int_{\RR^n} f(y)\phi^m\left(\frac{x-y-s}{r} \right)\left(1+\frac{|s|}{r} \right)^{-N}\frac{dy}{r^n}\right|
\\&\lesssim&\nonumber   \sum_{k=0}^{+\infty}2^{-(k-1)N}\sup_{ 0\leq |s|< 2^{k}r}\left|\int_{\RR^n} f(y)\phi^m\left(\frac{x-y-s}{r} \right)\frac{dy}{r^n}\right|.
\end{eqnarray}
Formula\,(\ref{no1}) leads to
\begin{eqnarray}\label{tan19}
& &\int_{\RR^n}\sup_{0\leq |s|< 2^{k}r}\left|\int_{\RR^n}f(y)\phi^m\left(\frac{x-y-s}{r}\right)\frac{dy}{r^n}\right|^pdx\\ \nonumber&\leq& C\left(1+2^k\right)^n\int_{\RR^n}\sup_{0\leq |s|<r}\left|\int_{\RR^n}f(y)\phi^m\left(\frac{x-y-s}{r}\right)\frac{dy}{r^n}\right|^pdx .
\end{eqnarray}
For $N>n/p$,  Formulas\,(\ref{tan27}, \ref{tan19}) lead to
\begin{eqnarray}\label{tan20}
 \int_{\RR^n} |M_{\psi}f(x)|^pdx\leq C \int_{\RR^n}\sup_{0\leq |s|<r}\left|\int_{\RR^n}f(y)\phi^m\left(\frac{x-y-s}{r}\right)\frac{dy}{r^n}\right|^pdx ,
 \end{eqnarray}
where $C$ is a constant independent on $m$. We use $\left(F^m f\right)(x, r)$ and $\left(F f\right)(x, r)$ to denote as following:
\begin{eqnarray*}
 \left(F^m f\right)(x, r)=\int_{\RR^n}f(y)\phi^m\left(\frac{x-y}{r}\right)\frac{dy}{r^n},\ \ \left(F f\right)(x, r)=\int_{\RR^n}f(y)\phi\left(\frac{x-y}{r}\right)\frac{dy}{r^n}.
\end{eqnarray*}
Thus we could have:
\begin{eqnarray}
\left|\left(F^m f\right)(u, r)-\left(Ff\right)(u, r)\right|&\leq& \int_{\RR^n} \left|f(y)\right|\left|\phi^m\left(\frac{u-y}{r}\right)-\phi\left(\frac{u-y}{r}\right)\right|\frac{dy}{r^n}
\\&\leq&\nonumber C \frac{1}{m}|M f(u)|,
\end{eqnarray}
where C is dependent on $\gamma$ and   $M$ is the Hardy-Littlewood Maximal Operator.
Let us set: $$\delta_m(u)=\left|\left(F^m f\right)(u, r)-\left(Ff\right)(u, r)\right|.$$
Thus we could deduce the following:
$$\left\{x: \delta_m(x)>\alpha\right\}\subseteq \left\{x: M f(x) > \frac{1}{c} m\alpha \right\}\ \ \ \hbox{for\,some\,constant\,c}.$$
Notice that $M$ is weak-(1, 1) bounded, thus the following holds for any $\alpha>0$:
$$\left|\left\{x: \delta_m(x)>\alpha\right\}\right|\leq \frac{c}{\alpha} \|f\|_{L^1(\RR^n)}\frac{1}{m}.$$
Thus we could obtain:
$$\lim_{m\rightarrow+\infty}\left|\left\{x: \delta_m(x)>\alpha\right\}\right|=0.$$
Thus there exists a sequence $\{m_j\}\subseteq \{m\}$ such that the following holds: $$\lim_{m_j\rightarrow+\infty}\left(F^{m_j}f\right)(u, r)=\left(Ff\right)(u, r),\ \ \ a.e.u\in\RR^n $$ for $f\in L^1(\RR^n).$
Let us set $E$ as: $$E=\{u\in\RR^n: \lim_{m_j\rightarrow+\infty}\left(F^{m_j}f\right)(u, r)=\left(Ff\right)(u, r)\}.$$
Thus it is clear that  E is  dense in $\RR^n$. For any $x_0\in\RR^n$, there exists a $(u_0, r_0)$ with $r_0>0$, $u_0\in\RR^n$, $|u_0- x_0|<r_0$ such that the following holds:
$$\left|\left(F^{m_j}f\right)(u_0, r_0)\right| \geq \frac{1}{2} \sup_{|x_0-u|<r}|\left(F^{m_j} f\right)(u, r)|.$$
Notice that $\left(F^{m_j}f\right)(u, r_0)$ is a continuous function in $u$ variable and E is  dense in $\RR^n$. There exists a $\widetilde{u}_0\in E$ with $|\widetilde{u}_0- x_0|<r_0$ such that
$$\left|\left(F^{m_j}f\right)(\widetilde{u}_0, r_0)\right| \geq \frac{1}{4} \sup_{|x_0-u|<r}|\left(F^{m_j} f\right)(u, r)|.$$
Thus we could deduce that
\begin{eqnarray}\label{tan23}
\sup_{\{u\in E: |x-u|<r\}} \left|\left(F^{m_j}f\right)(u, r)\right|\sim \sup_{\{u\in \RR^n: |x-u|<r\}} \left|\left(F^{m_j}f\right)(u, r)\right|.
\end{eqnarray}
Formula\,(\ref{tan23}) together with the dominated convergence theorem, we could conclude:
\begin{eqnarray}\label{tan24}
\overline{\lim}_{m_j\rightarrow+\infty}\int_{\RR^n} \sup_{|x-u|<r}\left|\left(F^{m_j} f\right)(u, r)\right|^pdx &\sim&\nonumber \overline{\lim}_{m_j\rightarrow+\infty}\int_{\RR^n} \sup_{\{u\in E: |u-x|<r\}} \left|\left(F^{m_j}f\right)(u, r)\right|^pdx
\\&\leq&\nonumber C \int_{\RR^n} \overline{\lim}_{m_j\rightarrow+\infty} \sup_{\{u\in E: |u-x|<r\}} \left|\left(F^{m_j}f\right)(u, r)\right|^pdx
\\&\leq&\nonumber C  \int_{\RR^n}\sup_{\{u\in E: |u-x|<r\}} \left|\left(Ff\right)(u, r)\right|^pdx
\\&\leq& C  \int_{\RR^n}\sup_{\{u\in \RR^n: |u-x|<r\}} \left|\left(Ff\right)(u, r)\right|^pdx.
\end{eqnarray}
Also it is clear that  the following inequality holds for $f\in L^1(\RR^n)$:
\begin{eqnarray}\label{tan26}
\|(f*\phi)_{\nabla}\|_{L^p(\RR^n)}\lesssim \|f_{\gamma}^*\|_{L^p(\RR^n)}=\|f\|_{H^p(\RR^n)}.
\end{eqnarray}
Notice that $H^p(\RR^n)\bigcap L^1(\RR^n)$ is dense in $H^p(\RR^n)$. Then by Formula\,(\ref{tan24}) and Formula\,(\ref{tan26}),  we could deduce that the following inequality holds for  $f\in H^p(\RR^n)$
\begin{eqnarray}\label{tan25}
\|f\|_{H^p(\RR^n)}\sim_{p, \gamma, \phi}\|(f*\phi)_{\nabla}\|_{L^p(\RR^n)} .
\end{eqnarray}
This proves our proposition.
\end{proof}

\begin{proposition}\label{pp2}
For $\phi(x)\in \Lambda^{\gamma}$, $supp\,\phi(x)\subseteq\left\{x\in\RR^n: |x|<1\right\}$, $\frac{1}{1+\gamma} <p\leq1$, $|\phi(x)| \leq1$  $\int\phi(x)dx=1$, then we could obtain the following inequality for  $f\in H^p(\RR^n)$:
\begin{eqnarray}\label{tan27o}
\|(f*\phi)_{\nabla}\|_{L^p(\RR^n)}\sim_{\gamma,  \phi, p} \|(f*\phi)_{+}\|_{L^p(\RR^n)}.
\end{eqnarray}

\end{proposition}

\begin{proof}
Let us set $0<\alpha<\gamma\leq1$, $f\in L^1(\RR^n)$ and $\displaystyle{\ 1\geq p> \frac{1}{1+\gamma-\alpha}}$ first.
We use $F(a, b, x, y, r)$ with $|a-b|<r$ to denote as:
\begin{eqnarray*}
F(a, b, y, z, r)=\left(\phi\left(\frac{a-y}{r}\right)-\phi\left(\frac{b-y}{r}\right)\right)-\left(\phi\left(\frac{a-z}{r}\right)-\phi\left(\frac{b-z}{r}\right)\right).
\end{eqnarray*}
We use $T(a, b, y, r)$ with $|a-b|<r$ to denote as:
\begin{eqnarray*}
T(a, b, y,  r)=\phi\left(\frac{a-y}{r}\right)-\phi\left(\frac{b-y}{r}\right).
\end{eqnarray*}
Thus it is clear to see that the following inequalities hold for $0<\alpha<\gamma\leq1$:
\begin{eqnarray}\label{ut1}
& \displaystyle{|F(a, b, y, z, r)|\leq C \left(\frac{a-b}{r}\right)^{\alpha}\left(\frac{y-z}{r}\right)^{\gamma-\alpha}}; \\ \label{ut2}
& \displaystyle{|T(a, b, y,  r)|\leq C;}\\ \label{ut6}
& supp\,T(a, b, y,  r)\subseteq B(x, 2r)\ \hbox{when}\,a\in B(x, r),\,\,|a-b|<r.
\end{eqnarray}
For $\displaystyle{\ p> \frac{1}{1+\gamma-\alpha}}$, let $F$ denote as: $$F=\bigg\{x\in\RR^n: f_{\gamma-\alpha}^*(x)\leq\sigma (f*\phi)_{\nabla}(x)\bigg\}.$$
It is clear  that the following inequality  holds for  $f\in L^1(\RR^n)\bigcap H^p(\RR^n)$:
$$\|f_{\gamma-\alpha}^*\|_{L^p(\RR^n)}\sim_{\gamma, \alpha}\|f_{S(\gamma-\alpha)}^*\|_{L^p(\RR^n)}\sim_{\gamma, \alpha}\|f_{S\gamma}^*\|_{L^p(\RR^n)}\sim_{\gamma, \alpha}\|f_{\gamma}^*\|_{L^p(\RR^n)}.$$
Then by Proposition\,\ref{pp1}, we could obtain
\begin{eqnarray}\label{ut5}
\int_{F^c} |(f*\phi)_{\nabla}(x)|^pdx &\leq& C\frac{1}{\sigma^p} \int_{F^c}| f_{\gamma-\alpha}^*(x)|^p dx\leq \frac{C_{\gamma, \alpha}'}{\sigma^p} \int_{\RR^n}| f_{\gamma}^*(x)|^p dx\\ \nonumber&\leq& \frac{C_{\gamma, \alpha, \phi}''}{\sigma^p} \int_{\RR^n}|(f*\phi)_{\nabla}(x)|^p dx.
\end{eqnarray}\\
Choosing $\sigma^p\geq2C_{\gamma, \alpha, \phi}''$, we could have
\begin{eqnarray}\label{ut3}
 \int_{\RR^n}|(f*\phi)_{\nabla}(x)|^p dx\lesssim_{\gamma, \alpha, \phi}\int_{F} |(f*\phi)_{\nabla}(x)|^pdx.
\end{eqnarray}
Denote $D f(x)$ and $F(x, r)$ as: $$ D f(x)=\sup_{t>0}\left|f*\phi_t(x)\right|,\ \ \ F(x, t)=f*\phi_t(x).$$
Next, we will show that for any $q>0$,
\begin{eqnarray}\label{ut4}
(f*\phi)_{\nabla}(x)\leq C \left[M\left(D f\right)^q(x)\right]^{1/q} \ \ \hbox{for}\ x\in F,
\end{eqnarray}
where $M$ is the Hardy-Littlewood maximal operator. Fix any $x_0\in F$, then there exists $(u_0, r_0)$ satisfying $|u_0-x_0|<r_0$ such that  the following inequality holds:
\begin{eqnarray}\label{abc1}
 \left|F(u_0, r_0)\right|>\frac{1}{2} (f*\phi)_{\nabla}(x_0).
\end{eqnarray}
Choosing $\delta<1$ small enough and  $u$ with $|u-u_0|< \delta r_0$,  we could deduce that
\begin{eqnarray*}
|F(u, r_0)-F(u_0, r_0)|&=&\left|\int_{\RR^n} f(y)\phi\left(\frac{u-y}{r_0}\right)\frac{dy}{r_0^n} -\int_{\RR^n} f(y)\phi\left(\frac{u_0-y}{r_0}\right)\frac{dy}{r_0^n}  \right|
\\&\leq& \left|  \int_{\RR^n} f(y)  T(u, u_0, y,  r_0)\frac{dy}{r_0^n}                 \right|.
\end{eqnarray*}
We could consider $  T(u, u_0, y,  r_0)$ as a new kernel. By Formulas\,(\ref{ut1}, \ref{ut2}, \ref{ut6}) we could obtain:
\begin{eqnarray*}
|F(u, r_0)-F(u_0, r_0)|
\leq C\delta^{\alpha} f_{\gamma-\alpha}^*(x_0)\leq C\delta^{\alpha}\sigma (f*\phi)_{\nabla}(x_0) \ \ \hbox{for}\ x_0\in F.
\end{eqnarray*}
Taking $\delta$ small enough such that $ C\delta^{\alpha}\sigma\leq1/4$, we could obtain $$|F(u, r_0)|\geq \frac{1}{4}(f*\phi)_{\nabla}(x_0) \ \ \hbox{for}\ u\in I(u_0, \delta r_0 ).$$
Thus the following inequality holds for any $x_0\in F$,
\begin{eqnarray*}
 \left|(f*\phi)_{\nabla}(x_0)  \right|^q &\leq& \left|\frac{1}{B(u_0, \delta r_0 )}\right| \int_{B(u_0, \delta r_0 )}  4^q|F(u, r_0)|^q du
 \\&\leq&   \left|\frac{B(x_0, (1+\delta) r_0 )}{B(u_0, \delta r )}\right|  \left|\frac{1}{B(x_0, (1+\delta) r_0 )}\right|\int_{I(x_0, (1+\delta) r_0)}  4^q|F(u, r_0)|^qdu
\\&\leq&\left(\frac{1+\delta}{\delta}\right)^n\left|\frac{1}{B(x_0, (1+\delta) r_0)}\right|\int_{B(x_0, (1+\delta) r_0)}  4^q|F(u, r_0)|^qdu
\\&\leq& C M[(D f)^q](x_0)
\end{eqnarray*}
C is independent on $x_0$. Finally, using the maximal theorem for $M$ when $q<p$ leads to
\begin{eqnarray}\label{abc3}
\int_{F} \left|(f*\phi)_{\nabla}(x)  \right|^p dx\leq C \int_{\RR^n}\left\{M[(D f)^{q}](x)\right\}^{p/q}dx \leq C \int_{\RR^n}\left|(f*\phi)_{+}(x)\right|^{p}dx.
\end{eqnarray}
Thus for any fixed $\alpha$ satisfying $ 0<\alpha<\gamma$ and $\displaystyle{\ 1\geq p> \frac{1}{1+\gamma-\alpha}}$, the above Formula\,(\ref{abc3}) combined with Formula\,(\ref{ut3}) lead to
\begin{eqnarray}\label{abc4}
\|(f*\phi)_{\nabla}\|_{L^p(\RR^n)}\leq C\|(f*\phi)_{+}\|_{L^p(\RR^n)}\ \ ,
\end{eqnarray}
where $C$ is dependent on $p$ and $\alpha$. Next we will remove the number $\alpha$.  For any $\displaystyle{\ 1\geq p> \frac{1}{1+\gamma}}$, let $p_0=\frac{1}{2}\left(p+\frac{1}{1+\gamma}\right)$ with $\displaystyle{\ p> p_0>\frac{1}{1+\gamma}}$ and let $\alpha= 1+\gamma -\frac{1}{p_0}$. By Formula\,(\ref{abc4}),   we could obtain the following inequality holds for $\displaystyle{\ 1\geq p> \frac{1}{1+\gamma}}$ and $f\in L^1(\RR^n)$
$$\|(f*\phi)_{\nabla}\|_{L^p(\RR^n)}\leq C\|(f*\phi)_{+}\|_{L^p(\RR^n)}, $$
where $C$ is dependent on $p$ and $\gamma$. Thus by the fact that  $ L^1(\RR^n)\bigcap H^p(\RR^n) $ is dense in $  H^p(\RR^n) $, we could deduce
Formula\,(\ref{tan27}) holds for any $f\in H^p(\RR^n)$. This proves the Proposition.
\end{proof}
Thus from Proposition\,\ref{pp1} and Proposition\,\ref{pp2}, we could obtain the following theorem:
\begin{theorem}\label{pp3}
For $\phi(x)\in \Lambda^{\gamma}$, $supp\,\phi(x)\subseteq\left\{x\in\RR^n: |x|<1\right\}$, $1\geq p>\frac{1}{1+\gamma}$, $|\phi(x)| \leq1$  $\int\phi(x)dx=1$, then we could obtain the following inequality for any $f\in H^p(\RR^n)$:
\begin{eqnarray}\label{tan28}
\|f\|_{H^p(\RR^n)}\sim_{p, \gamma, \phi}\|(f*\phi)_{\nabla}\|_{L^p(\RR^n)}\sim_{ p,  \gamma, \phi} \|(f*\phi)_{+}\|_{L^p(\RR^n)}.
\end{eqnarray}
\end{theorem}
\begin{proposition}\label{pp4}
For $\phi(x)\in S(\RR^n)$, with $\left|\SF\phi(\xi)\right|\geq C$ for $\xi\in B(\xi_0, r_0)$,  we could obtain the following inequality for $f\in H^p(\RR^n)$:
\begin{eqnarray}\label{tan30}
\|f\|_{H^p(\RR^n)}\sim_{p, \xi_0, r_0, \phi}\|(f*\widetilde{\phi})_{\nabla}\|_{L^p(\RR^n)}\sim_{p, \xi_0, r_0, \phi} \|(f*\widetilde{\phi})_{+}\|_{L^p(\RR^n)}\ \ \hbox{when}\ 0<p\leq 1,
\end{eqnarray}
where $\widetilde{\phi}(x)=\phi(r_0x)e^{-2\pi i(\xi_0, x)}$.
\end{proposition}
\begin{proof} Notice that $\SF \widetilde{\phi}(\xi)=\SF \phi\left(\frac{\xi+\xi_0}{r_0}\right)$, thus $\left|\SF\widetilde{\phi}(\xi)\right|\geq C$, when $\xi\in B(0, 1)$.  Then  we can prove this proposition in a way similar to Proposition\,\ref{pp1}.
\end{proof}
Thus in a way  similar to Theorem\,\ref{pp3}, by Proposition\,\ref{pp4}, we could obtain the following corollary:

\begin{corollary}\label{pp5}
For $\phi(x)\in \Lambda^{\gamma}$, $supp\,\phi(x)\subseteq\left\{x\in\RR^n: |x|<1\right\}$, $1\geq p>\frac{1}{1+\gamma}$, $|\phi(x)| \leq1$, with $|\SF\phi(\xi)|\geq C$, when $\xi\in B(\xi_0, r_0)$, then we could obtain the following inequality for any $f\in H^p(\RR^n)$:
\begin{eqnarray*}
\|f\|_{H^p(\RR^n)}\sim_{p, \gamma, \xi_0, r_0, \phi}\|(f*\widetilde{\phi})_{\nabla}\|_{L^p(\RR^n)}\sim_{ p,  \gamma, \xi_0, r_0, \phi} \|(f*\widetilde{\phi})_{+}\|_{L^p(\RR^n)},
\end{eqnarray*}
where $\widetilde{\phi}(x)=\phi(r_0x)e^{-2\pi i(\xi_0, x)}$.
\end{corollary}

\section{ Lipschitz function withot compact support in $\RR^n$}
\begin{proposition}\label{lp51}
$\phi(x)$ is a Lipschitz function ($\phi(x)\in \Lambda^{\gamma}$) without compact support in $\RR^n$ satisfying the following:
\begin{eqnarray}\label{lp1}
\left|\phi(x)\right|\lesssim \frac{1}{(1+|x|)^{n+\gamma}},
\end{eqnarray}
\begin{eqnarray}\label{lp2}
\left|\phi(x+h)-\phi(x)\right|\lesssim \frac{|h|^{\gamma}}{(1+|x|)^{n+2\gamma}},\ \ \hbox{if}\ |h|\lesssim1+|x|.
\end{eqnarray}
Then we could deduce the following inequality for $\frac{1}{1+\gamma}<p\leq1$:
$$\|(f*\phi)_{\nabla}\|_{L^p(\RR^n)}\lesssim_{p, n, \gamma}\|f\|_{H^p(\RR^n)}.$$
\end{proposition}
\begin{proof}
Fix a positive $\varphi(t)\in S(\RR^n)$ so that  $supp\,\varphi(t)\subseteq B(0, 1)$, and $\varphi(t)=1$ for $t\in B(0, 1/2)$. Let
the functions $\psi_{k,x}(t)$ be defined as follows:
\begin{eqnarray*}
& &\psi_{0,x}(t)=\varphi\left(\frac{x-t}{r}\right), \\ \nonumber
& &\psi_{k,x}(t)=\varphi\left(\frac{x-t}{2^{k}r}\right)-\varphi\left(\frac{x-t}{2^{k-1}r}\right),\ \hbox{for}\  k\geq1.
\end{eqnarray*}
Thus $\psi_{k,x}(t)\in S(\RR^n) $ for $k \geq0$, with $supp\,\psi_{0,x}(t)\subseteq B(x, r)$, $supp\,\psi_{k,x}(t)\subseteq B(x, 2^{k+1}r)\setminus B(x, 2^{k-2}r) \,\hbox{for}\,k\geq1$. It is clear that $$\sum_{k=0}^{\infty}\psi_{k,x}(t)=1.$$
Then we could write $(f*\phi)_{\nabla}(x)$ as following:
\begin{eqnarray*}
(f*\phi)_{\nabla}(x)&=&\sup_{ |s-x|\leq r}\left|\int_{\RR^n}\phi\left(\frac{s-y}{r}\right)\sum_{k=0}^{\infty}\psi_{k,x}(y)f(y)dy/r^n\right|
\\&\leq& \sum_{k=0}^{+\infty}\sup_{ |s-x|\leq r }\left|\int_{\RR^n}\phi\left(\frac{s-y}{r}\right) \psi_{k,x}(y) f(y)dy/r^n\right|.
\end{eqnarray*}

It is clear that the function $y\rightarrow (1+2^k)^{n+\gamma}\phi\left(\frac{s-y}{r}\right)\psi_{k,x}(y)$ with $|s-x|<r$ satisfies the following:
\begin{eqnarray*}
\left\{ \begin{array}{cc}
 \left|(1+2^k)^{n+\gamma}\phi\left(\frac{s-y}{r}\right)\psi_{k,x}(y)\right|\lesssim1                                \\\\
  H^{\gamma}\left((1+2^k)^{n+\gamma}\phi\left(\frac{s-y}{r}\right)\psi_{k,x}(y)\right)\lesssim \left(2^kr\right)^{-\gamma}                            \\\\
   supp(1+2^k)^{n+\gamma}\phi\left(\frac{s-y}{r}\right)\psi_{k,x}(y)\subseteq B(x, 2^{k+1}r)\setminus B(x, 2^{k-2}r) \,\hbox{for}\,k\geq1.
 \end{array} \right.
 \end{eqnarray*}
Then we could deduce that:
\begin{eqnarray*}
(f*\phi)_{\nabla}(x)&=&\sup_{|s-x|\leq r}\left|\int_{\RR^n}\phi\left(\frac{s-y}{r}\right)f(y)dy/r^n\right|
\\&\leq& \sum_{k=0}^{+\infty}\frac{(2^{k})^n}{(1+2^k)^{n+\gamma}}\sup_{|s-x|\leq r}\left|\int_{\RR^n}(1+2^k)^{n+\gamma}\phi\left(\frac{s-y}{r}\right)\psi_{k,x}(y)f(y)dy/(2^{k}r)^n\right|
\\&\lesssim& \sum_{k=0}^{+\infty} \frac{(2^{k})^n}{(1+2^k)^{n+\gamma}}f_{\gamma}^*(x)
\\&\lesssim_{n, \gamma}& f_{\gamma}^*(x).
\end{eqnarray*}
Thus the following inequality holds for $\frac{1}{1+\gamma}<p\leq1$ $(0<\gamma\leq1)$:
$$\|(f*\phi)_{\nabla}\|_{L^p(\RR^n)}\lesssim_{n, p, \gamma}\|f\|_{H^p(\RR^n)}.$$
This proves the proposition.
\end{proof}

\begin{proposition}\label{kernel1}

$\phi(x)$ is a Lipschitz function ($\phi(x)\in \Lambda^{\gamma}$) without compact support in $\RR^n$ satisfying the following:
\begin{eqnarray}\label{lp6}
\left|\phi(x)\right|\lesssim \frac{1}{(1+|x|)^{n+\gamma}},
\end{eqnarray}
\begin{eqnarray}\label{lp7}
\left|\phi(x+h)-\phi(x)\right|\lesssim \frac{|h|^{\gamma}}{(1+|x|)^{n+2\gamma}},\ \ \hbox{if}\ |h|\lesssim1+|x|.
\end{eqnarray}
Then we could deduce the following inequalities for any fixed $\alpha$ with $ 0<\alpha<\gamma\leq1,$ and $r>0$:
$$0\le \left|\phi\left(\frac{a-y}{r}\right)-\phi\left(\frac{b-y}{r}\right)\right| \le
C \Big(\frac{|a-b|}{r}\Big)^{\alpha}\Big(1+\frac{|x-y|}{r}\Big)^{-(\gamma-\alpha)-n}, $$
and
\begin{eqnarray*}
& &\left|\left(\phi\left(\frac{a-y}{r}\right)-\phi\left(\frac{b-y}{r}\right)\right)-\left(\phi\left(\frac{a-z}{r}\right)-\phi\left(\frac{b-z}{r}\right)\right)\right|
\\&\leq & C \Big(\frac{|a-b|}{r}\Big)^{\alpha}\Big(\frac{|y-z|}{r}\Big)^{\gamma-\alpha}\Big(1+\frac{|x-y|}{r}\Big)^{-2(\gamma-\alpha)-n},
\end{eqnarray*}
for $|a-
b|\lesssim r $, $\frac{|y-
z|}{r}\leq C_3\min\{1+\frac{|a-y|}{r}, 1+\frac{|a-z|}{r}\} $, $x\in B(a, 2r)\bigcap B(b, 2r)$.
\end{proposition}
\begin{proof}
From the fact that $|a-
b|\lesssim r $, $\frac{|y-
z|}{r}\leq C_3\min\{1+\frac{|a-y|}{r}, 1+\frac{|a-z|}{r}\} $, the following relations could be obtained:
\begin{eqnarray}\label{1120}
1+\frac{|a-y|}{r}\sim1+\frac{|b-y|}{r}, 1+\frac{|a-z|}{r}\sim1+\frac{|b-z|}{r},\  and\, 1+\frac{|a-z|}{r}\sim1+\frac{|a-y|}{r}.
\end{eqnarray}
First,  we will consider the case when $|a- b|\leq |y- z|.$ Then from Formula\,(\ref{lp7}), we could get
\begin{eqnarray}\label{1119}
\left|\phi\left(\frac{a-y}{r}\right)-\phi\left(\frac{b-y}{r}\right)\right| &\le&
C\Big(\frac{|a- b|}{r}\Big)^{\gamma}\Big(1+\frac{|a- y|}{r}\Big)^{-2\gamma-n}
\\&\le&\nonumber
C\Big(\frac{|a- b|}{r}\Big)^{\gamma}\Big(1+\frac{|a- y|}{r}\Big)^{-\gamma-\alpha}\Big(1+\frac{|a- y|}{r}\Big)^{-(\gamma-\alpha)-n}
\\&\le&\nonumber
C \Big(\frac{|a- b|}{r}\Big)^{\alpha}\Big(1+\frac{|a- y|}{r}\Big)^{-(\gamma-\alpha)-n}.
\end{eqnarray}
Also we could obtain
$$\left|\phi\left(\frac{a-y}{r}\right)-\phi\left(\frac{b-y}{r}\right)\right| \le
C\Big(\frac{|a- b|}{r}\Big)^{\gamma}\Big(1+\frac{|a- y|}{r}\Big)^{-2\gamma-n},$$
and
$$\left|\phi\left(\frac{a-z}{r}\right)-\phi\left(\frac{b-z}{r}\right)\right| \le
C\Big(\frac{|a- b|}{r}\Big)^{\gamma}\Big(1+\frac{|a- z|}{r}\Big)^{-2\gamma-n}.$$
Together with Formula\,(\ref{1120}), we could conclude
\begin{eqnarray*}
& &\left|\left(\phi\left(\frac{a-y}{r}\right)-\phi\left(\frac{b-y}{r}\right)\right)-\left(\phi\left(\frac{a-z}{r}\right)-\phi\left(\frac{b-z}{r}\right)\right)\right|
\\&\leq & C\Big(\frac{|a- b|}{r}\Big)^{\gamma}\Big(1+\frac{|a- y|}{r}\Big)^{-2\gamma-n}.
\end{eqnarray*}
By the fact $|a-b|\leq |y-z|$ and $\displaystyle{1\lesssim 1+\frac{|a-y|}{r}}$, we could obtain:
$$\Big(\frac{|a-b|}{r}\Big)^{\gamma}\Big(1+\frac{|a-y|}{r}\Big)^{-2\gamma-n}\lesssim\Big(\frac{|a-b|}{r}\Big)^{\alpha}\Big(\frac{|y-z|}{r}\Big)^{\gamma-\alpha}\Big(1+\frac{|a-y|}{r}\Big)^{-2(\gamma-\alpha)-n}.$$
Then for  $|a-b|\leq |y-z|,$ the Formula
\begin{eqnarray}\label{exu1}
& &\left|\left(\phi\left(\frac{a-y}{r}\right)-\phi\left(\frac{b-y}{r}\right)\right)-\left(\phi\left(\frac{a-z}{r}\right)-\phi\left(\frac{b-z}{r}\right)\right)\right|\nonumber
\\&\leq & C \Big(\frac{|a-b|}{r}\Big)^{\alpha}\Big(\frac{|y-z|}{r}\Big)^{\gamma-\alpha}\Big(1+\frac{|a-y|}{r}\Big)^{-2(\gamma-\alpha)-n}
\end{eqnarray}
holds.
In a similar way, we will obtain the Formula\,(\ref{exu1}) for the case when $|a-b|\geq |y-z|.$
Notice that by Formula\,(\ref{1120}),
$$\left|\phi\left(\frac{a-y}{r}\right)-\phi\left(\frac{a-z}{r}\right)\right|\le
C\Big(\frac{|y-z|}{r}\Big)^{\gamma}\Big(1+\frac{|a-y|}{r}\Big)^{-2\gamma-n},$$
and
\begin{eqnarray*}
\left|\phi\left(\frac{b-y}{r}\right)-\phi\left(\frac{b-z}{r}\right)\right|&\le&
C\Big(\frac{|y-z|}{r}\Big)^{\gamma}\Big(1+\frac{|b-y|}{r}\Big)^{-2\gamma-n}
\\&\le&
C\Big(\frac{|y-z|}{r}\Big)^{\gamma}\Big(1+\frac{|a-y|}{r}\Big)^{-2\gamma-n}
\end{eqnarray*}
hold. Then we could obtain
\begin{eqnarray*}
& &\left|\left(\phi\left(\frac{a-y}{r}\right)-\phi\left(\frac{b-y}{r}\right)\right)-\left(\phi\left(\frac{a-z}{r}\right)-\phi\left(\frac{b-z}{r}\right)\right)\right|
\\&\leq & C \Big(\frac{|y-z|}{r}\Big)^{\gamma}\Big(1+\frac{|a-y|}{r}\Big)^{-2\gamma-n}.
\end{eqnarray*}
By the fact $|a-b|\geq |y-z|$ and $\displaystyle{1\lesssim 1+\frac{|a-y|}{r}}$, the following holds:
$$\Big(\frac{|y-z|}{r}\Big)^{\gamma}\Big(1+\frac{|a-y|}{r}\Big)^{-2\gamma-n}\lesssim\Big(\frac{|a-b|}{r}\Big)^{\alpha}\Big(\frac{|y-z|}{r}\Big)^{\gamma-\alpha}\Big(1+\frac{|a-y|)}{r}\Big)^{-2(\gamma-\alpha)-n}.$$
Then for  $|a-b|\geq |y-z|,$ we could get
\begin{eqnarray}\label{exu2}
& &\left|\left(\phi\left(\frac{a-y}{r}\right)-\phi\left(\frac{b-y}{r}\right)\right)-\left(\phi\left(\frac{a-z}{r}\right)-\phi\left(\frac{b-z}{r}\right)\right)\right|\nonumber
\\&\leq & C \Big(\frac{|a-b|}{r}\Big)^{\alpha}\Big(\frac{|y-z|}{r}\Big)^{\gamma-\alpha}\Big(1+\frac{|a-y|}{r}\Big)^{-2(\gamma-\alpha)-n}.
\end{eqnarray}
By the fact that $x\in B(a, 2r)\bigcap B(b, 2r)$, we could deduce that:
\begin{eqnarray}\label{1121}
1+\frac{|a-y|}{r}\sim1+\frac{|x-y|}{r}.
\end{eqnarray}
Formulas\,(\ref{1119}, \ref{exu1}, \ref{exu2}, \ref{1121})\,yeald the Proposition.
\end{proof}

\begin{proposition}\label{no11}
For  $\displaystyle{\ 1\geq p> \frac{1}{1+\gamma}}$, $\phi(x)$ is a Lipschitz function ($\phi(x)\in \Lambda^{\gamma}$) without compact support in $\RR^n$ satisfying Formulas\,(\ref{lp6}, \ref{lp7}). For $f\in L^1(\RR^n)$,  if the following inequality holds
$$\|(f*\phi)_{\nabla}\|_{L^p(\RR^n)}\sim\|f_{\gamma}^* \|_{L^p(\RR^n)}$$
then  we could deduce  that:
$$\|(f*\phi)_{\nabla}\|_{L^p(\RR^n)}\leq C\|(f*\phi)_{+} \|_{L^p(\RR^n)},$$
where $C$ is dependent on $p$ and $\gamma$.
\end{proposition}
\begin{proof}
By Proposition\,\ref{sk}, we could deduce that the following holds for  $f\in L^1(\RR^n)\bigcap H^p(\RR^n)$:
$$\|f_{\gamma-\alpha}^*\|_{L^p(\RR^n)}\sim_{\gamma, \alpha}\|f_{S(\gamma-\alpha)}^*\|_{L^p(\RR^n)}\sim_{\gamma, \alpha}\|f_{S\gamma}^*\|_{L^p(\RR^n)}\sim_{\gamma, \alpha}\|f_{\gamma}^*\|_{L^p(\RR^n)}.$$
For any fixed $\alpha$ satisfying $ 0<\alpha<\gamma$ and $\displaystyle{\ 1\geq p> \frac{1}{1+\gamma-\alpha}}$. We use $F$ to denote as: $$F=\bigg\{x\in\RR^n: f_{\gamma-\alpha}^*(x)\leq\sigma (f*\phi)_{\nabla}(x)\bigg\}.$$
Then it is clear that
\begin{eqnarray}\label{sv1}
\int_{F^c} |(f*\phi)_{\nabla}(x)|^pdx\leq \frac{C}{\sigma^p} \int_{F^c}| f_{\gamma-\alpha}^*(x)|^p dx\leq \frac{C_{\gamma, \alpha}'}{\sigma^p} \int_{\RR^n}| f_{\gamma}^*(x)|^p dx\leq \frac{C_{\gamma, \alpha}''}{\sigma^p} \int_{\RR^n}|(f*\phi)_{\nabla}(x)|^p dx.
\end{eqnarray}\\
Choosing $\sigma^p\geq2C_{\gamma, \alpha}''$, then the following holds:
\begin{eqnarray}\label{ex5}
 \int_{\RR^n}|(f*\phi)_{\nabla}(x)|^p dx\lesssim\int_{F} |(f*\phi)_{\nabla}(x)|^pdx.
\end{eqnarray}
We use $D f(x)$ and $F(x, r)$ to denote  as: $$ D f(x)=\sup_{t>0}\left|f*\phi_t(x)\right|,\ \ \ F(x, t)=f*\phi_t(x).$$
Next, we will show that for any $q>0$,
\begin{eqnarray}\label{ex6}
(f*\phi)_{\nabla}(x)\leq C \left[M\left(D f\right)^q(x)\right]^{1/q} \ \ \hbox{for}\ x\in F,
\end{eqnarray}
where $M$ is the Hardy-Littlewood maximal operator. For any fixed $x_0\in F$,  there exists $(u_0, r_0)$ satisfying $|u_0-x_0|<r_0$ such that  the following inequality holds:
\begin{eqnarray}\label{ex7}
 \left|F(u_0, r_0)\right|>\frac{1}{2} (f*\phi)_{\nabla}(x_0).
\end{eqnarray}
Choosing $\delta<1$ small enough and  $u$ satisfying $|u-u_0|< \delta r_0$,  we could deduce that
\begin{eqnarray*}
|F(u, r_0)-F(u_0, r_0)|&=&\left|\int_{\RR^n} f(y)\phi\left(\frac{u-y}{r_0}\right)\frac{dy}{r_0^n} -\int_{\RR^n} f(y)\phi\left(\frac{u_0-y}{r_0}\right)\frac{dy}{r_0^n}  \right|
\\&\leq& \left|  \int_{\RR^n} f(y)  \left(\phi\left(\frac{u-y}{r_0}\right)- \phi\left(\frac{u_0-y}{r_0}\right)\right)\frac{dy}{r_0^n}                 \right|.
\end{eqnarray*}
Notice that $\left(\phi\left(\frac{u-y}{r_0}\right)- \phi\left(\frac{u_0-y}{r_0}\right)\right)$ is a new kernel, thus by Proposition\,\ref{kernel1} and Proposition\,\ref{lp51}, we could obtain:
\begin{eqnarray*}
|F(u, r_0)-F(u_0, r_0)|
\leq C\delta^{\alpha} f_{\gamma-\alpha}^*(x_0)\leq C\delta^{\alpha}\sigma (f*\phi)_{\nabla}(x_0) \ \ \hbox{for}\ x_0\in F.
\end{eqnarray*}
Taking $\delta$ small enough such that $ C\delta^{\alpha}\sigma\leq1/4$, then $$\left|F(u, r_0)\right|\geq \frac{1}{4}(f*\phi)_{\nabla}(x_0) \ \ \hbox{for}\ u\in B(u_0, \delta r_0 ).$$
Thus the following inequality holds for any $x_0\in F$,
\begin{eqnarray*}
 \left|(f*\phi)_{\nabla}(x_0) \right|^q &\leq& \left|\frac{1}{B(u_0, \delta r_0 )}\right| \int_{B(u_0, \delta r_0 )}  4^q|F(u, r_0)|^q du
 \\&\leq&   \left|\frac{B(x_0, (1+\delta) r_0 )}{B(u_0, \delta r )}\right|  \left|\frac{1}{B(x_0, (1+\delta) r_0 )}\right|\int_{B(x_0, (1+\delta) r_0)}  4^q|F(u, r_0)|^qdu
\\&\leq&\left(\frac{1+\delta}{\delta}\right)^n\left|\frac{1}{B(x_0, (1+\delta) r_0)}\right|\int_{B(x_0, (1+\delta) r_0)}  4^q|F(u, r_0)|^qdu
\\&\leq& C M[(D f)^q](x_0),
\end{eqnarray*}
where C is a constant independent on $x_0$. Finally, using the maximal theorem for $M$ when $q<p$ leads to
\begin{eqnarray}\label{3}
\int_{F} \left|(f*\phi)_{\nabla}(x) \right|^p dx\leq C \int_{\RR^n}\left\{M[(D f)^{q}](x)\right\}^{p/q}dx \leq C \int_{\RR^n}\left|(f*\phi)_{\nabla}(x)\right|^{p}dx.
\end{eqnarray}
Thus for any fixed $\alpha$ satisfying $ 0<\alpha<\gamma$ and $\displaystyle{\ p> \frac{1}{1+\gamma-\alpha}}$, the above Formula\,(\ref{3}) combined with Formula\,(\ref{ex5}) leads to
\begin{eqnarray}\label{4}
\|(f*\phi)_{\nabla}\|_{L^p(\RR^n)}\leq C\|(f*\phi)_{+} \|_{L^p(\RR^n)},
\end{eqnarray}
where $C$ is dependent on $p$ and $\alpha$. Next we will remove the number $\alpha$.  For any $\displaystyle{\ 1\geq p> \frac{1}{1+\gamma}}$, let $p_0=\frac{1}{2}\left(p+\frac{1}{1+\gamma}\right)$ with $\displaystyle{\ p> p_0>\frac{1}{1+\gamma}}$ and let $\alpha= 1+\gamma -\frac{1}{p_0}$. Thus it is clear that $$p_0=\frac{1}{1+\gamma-\alpha},\ \ \ p>p_0.$$  Thus by Formula\,(\ref{4}),   we could obtain the following inequality holds for $\displaystyle{\ 1\geq p> \frac{1}{1+\gamma}}$
$$\|(f*\phi)_{\nabla}\|_{L^p(\RR^n)}\leq C\|(f*\phi)_{+} \|_{L^p(\RR^n)},$$
where $C$ is dependent on $p$ and $\gamma$. This proves the Proposition.
\end{proof}

\begin{proposition}\label{exx5}
For $N>[\frac{n}{p}]+1$ ($0<p\leq1$), $\phi(x)$ is a Lipschitz function ($\phi(x)\in \Lambda^{\gamma}$) without compact support in $\RR^n$ satisfying:
\begin{eqnarray}\label{up1}
\left|\phi(x)\right|\lesssim \frac{1}{(1+|x|)^{n+N+1}},
\end{eqnarray}
\begin{eqnarray}\label{up2}
\left|\phi(x+h)-\phi(x)\right|\lesssim \frac{|h|^{\gamma}}{(1+|x|)^{n+2\gamma}},\ \ \hbox{if}\ |h|\lesssim1+|x|.
\end{eqnarray}
Then there exists sequence $\{\phi^k(x): \phi^k(x)\in C_c(\RR^n)\}_{k=1}^{+\infty}$ such that:\\
{\rm(i)} \ $ supp\,\phi^k(x) \subseteq B(0, k)$;\\
{\rm(ii)} $ \lim_{k\rightarrow+\infty}\| \phi^k(x)-\phi(x)\|_{\infty} =0 $;\\
{\rm(iii)} $$\left|\phi^k(x)\right|\leq C_1 \frac{1}{(1+|x|)^{n+N+1}};$$
{\rm(iv)}  If $|h|\lesssim1+|x|$, then
$$
\left|\phi^k(x+h)-\phi^k(x)\right|\leq C_2 \frac{|h|^{\gamma}}{(1+|x|)^{n+2\gamma}};$$
{\rm(v)} The following two inequalities hold:
\begin{eqnarray*}
|\phi^k(x)-\phi(x)|\leq C_3 \left(\frac{1}{k}\right)^{\gamma/2} \frac{1}{(1+|x|)^{n+N+1-\frac{\gamma}{2}}},
\end{eqnarray*}
\begin{eqnarray*}
\int_{\RR^n}|\phi^k(x)-\phi(x)|dx\lesssim_\gamma \left(\frac{1}{k}\right)^{\gamma/2};
\end{eqnarray*}
where $C_1$, $C_2$, and $C_3$ are constants independent on $k$.
\end{proposition}
\begin{proof}
 Let $\psi(t)\in S(\RR^n)$ to be fixed satisfying $0<\psi(t)\leq1, \|H^{\gamma}\psi\|_{L^{\infty}}\leq C $, $supp\,\psi(t)\subseteq B(0, 1)$, $\psi(t)=1 $ when $t\in B(0, 1/2)$. The function $\phi^k(x)$ is  defined as: $$\phi^k(x)=\phi(x)\psi\left(\frac{x}{k}\right),\ \ \ k=1,\,2,\,3,\, \cdots.$$
Then it is clear that the sequence $\{\phi^k(x)\}_{k=1}^{+\infty}$ satisfies  {\rm(i)}, {\rm(ii)}, {\rm(iii)}, {\rm(iv)}.  When $|x|\leq \frac{k}{2}$, $|\phi^k(x)-\phi(x)|=0.$  When $|x|\geq \frac{k}{2}$,
\begin{eqnarray*}
|\phi^k(x)-\phi(x)|\le
C \frac{1}{(1+|x|)^{n+N+1}}
\leq C  \left(\frac{1}{k}\right)^{\gamma/2} \frac{1}{(1+|x|)^{n+N+1-\frac{\gamma}{2}}}.
\end{eqnarray*}
Then
\begin{eqnarray}\label{exexu1}
|\phi^k(x)-\phi(x)|\leq C_3 \left(\frac{1}{k}\right)^{\gamma/2} \frac{1}{(1+|x|)^{n+N+1-\frac{\gamma}{2}}}.
\end{eqnarray}
Thus
\begin{eqnarray}\label{exexu6}
\int_{\RR^n}|\phi^k(x)-\phi(x)|dx\lesssim_\gamma \left(\frac{1}{k}\right)^{\gamma/2}.
\end{eqnarray}
\end{proof}
By  Proposition\,\ref{exx5}, we could obtain the following Proposition:
\begin{proposition}\label{exx6}
$\phi(x)$ is a Lipschitz function ($\phi(x)\in \Lambda^{\gamma}$)($0<\gamma\leq1$) without compact support in $\RR^n$ satisfying Formulas\,(\ref{up1}, \ref{up2}). Then there exists sequence $\{\psi^k(x): \psi^k(x)\in S(\RR^n)\}_{k=1}^{\infty}$ satisfying the following:\\
{\rm(i)} \ $ supp\,\psi^k(x) \subseteq B(0, 2k)$;\\
{\rm(ii)} $ \lim_{k\rightarrow\infty}\| \psi^k(x)-\psi(x)\|_{\infty} =0 $;\\
{\rm(iii)}For $N\geq[\frac{n}{p}]+1$ ($0<p\leq1$), $\displaystyle{\left|\psi^k(x)\right|\leq C_1 \frac{1}{(1+|x|)^{n+N+1}};}$ \\
{\rm(iv)}  If $|h|\lesssim1+|x|$, then
$
\displaystyle{\left|\psi^k(x+h)-\psi^k(x)\right|\leq C_2 \frac{|h|^{\gamma}}{(1+|x|)^{n+2\gamma}};}$\\
{\rm(v)} The following two inequalities hold:
\begin{eqnarray*}
|\psi^k(x)-\phi(x)|\leq C_4 \left(\frac{1}{k}\right)^{\gamma/2} \frac{1}{(1+|x|)^{n+\frac{\gamma}{2}}},
\end{eqnarray*}
\begin{eqnarray*}
\int_{\RR^n}|\psi^k(x)-\phi(x)|dx\lesssim_\gamma \left(\frac{1}{k}\right)^{\gamma/2};
\end{eqnarray*}
where $C_1$, $C_2$, and $C_4$ are constants independent on $k$.
\end{proposition}
\begin{proof}
Let $\{\phi^k(x): \phi^k(x)\in C_c(\RR^n)\}_{k=1}^{\infty}$ to be the sequence  in Proposition\,\ref{exx5}. Let \begin{eqnarray*}
\rho(x)=\left\{ \begin{array}{cc}
                             \vartheta\exp\left\{\frac{1}{|x|^2-1}\right\}, \ \ \hbox{for}\  |x|<1\\\\
                             0, \ \ \ \hbox{for}\  |x|\geq1,
                           \end{array}\right.
\end{eqnarray*}
where $\vartheta$ is a constant such that $\int\rho(x)dx=1$.
Let
$$\phi^{k,\tau}(x)=\int_{\RR^n}  \phi^{k}(x-t)\rho\left(\frac{t}{\tau} \right)\frac{dt}{\tau^n}.$$
Thus {\rm(i)} {\rm(ii)} and {\rm(iii)} hold. Next we prove {\rm(iv)}.
Notice  that $supp\,\rho(x)\subseteq\{x: |x|<1\}$. Let us set $\tau=\frac{1}{k}$ for $k\in\ZZ, k\geq1$, thus it is clear that
$|h|\lesssim1+|x-t|$ holds if  $|h|\lesssim1+|x|$. Thus the following holds for $|h|\lesssim1+|x|$:
\begin{eqnarray}\label{ee5}
\left|\phi^{k,\frac{1}{k}}(x+h)-\phi^{k,\frac{1}{k}}(x))\right|&=&\left|\int_{\RR^n}  \phi^{k}(x+h-t)\rho\left(kt \right)k^ndt-\int_{\RR^n}  \phi^{k}(x-t)\rho\left(kt \right)k^ndt\right|
\\ \nonumber &\leq&\left|\int_{\RR^n}\left(\phi^{k}(x+h-t)-\phi^{k}(x-t)\right)\rho\left(kt \right)k^ndt\right|
\\ \nonumber &\leq&C_2 \frac{|h|^{\gamma}}{(1+|x|)^{n+2\gamma}}.
\end{eqnarray}
We could also deduce the following inequality:
\begin{eqnarray*}
|\phi^{k,\frac{1}{k}}(x)-\phi(x)|&\leq& |\phi^{k,\frac{1}{k}}(x)-\phi^k(x)|+|\phi^k(x)-\phi(x)|
\\ \nonumber &\leq&C_2 \frac{|\frac{1}{k}|^{\gamma}}{(1+|x|)^{n+2\gamma}}+ C_3 \left(\frac{1}{k}\right)^{\gamma/2} \frac{1}{(1+|x|)^{n+N+1-\frac{\gamma}{2}}}
\\ \nonumber &\leq&C_4 \left(\frac{1}{k}\right)^{\gamma/2} \frac{1}{(1+|x|)^{n+\frac{\gamma}{2}}}.
\end{eqnarray*}
Thus it is clear that
\begin{eqnarray*}
\int_{\RR^n}|\phi^{k,\frac{1}{k}}(x)-\phi(x)|dx\lesssim_\gamma \left(\frac{1}{k}\right)^{\gamma/2}.
\end{eqnarray*}
At last, we could set $\psi^k(x)$ as
$$\psi^k(x)=\phi^{k,\frac{1}{k}}(x).$$
This proves our proposition.
\end{proof}

\begin{proposition}\label{up0}
$\phi(x)$ is a Lipschitz function $(\phi(x)\in \Lambda^{\gamma},\,\int\phi(x)dx\sim1)$,\,$(0<\gamma\leq1), \phi(x)>0$ without compact support in $\RR^n$ satisfying Formulas\,(\ref{up1}, \ref{up2}), then we could deduce the following inequality for $\frac{1}{1+\gamma}<p\leq1$:
$$\|f\|_{H^p(\RR^n)}\lesssim_{p, n, \gamma, \phi} \|(f*\phi)_{\nabla}\|_{L^p(\RR^n)}.$$
\end{proposition}
\begin{proof}
Let  $f\in L^1(\RR^n)$ first. For $\phi(x)\in \Lambda^{\gamma}$, there exists sequence $\{\phi^m(x): \phi^m(x) \in S(\RR^n)\}_{m\in\NN}$ defined as Proposition\,\ref{exx6}.
It is clear that $\int\phi^m(x)dx\sim C$, where $C$ is a constant independent on $k$, thus $(\SF\phi^m)(2^{k_0}\xi)\geq C_p$, for $|\xi|\leq1$ where  $C_p$ is a constant independent on $m$, and $k_0$ is independent on $m$. Also by
Proposition\,\ref{exx6}, we could deduce that the following inequality holds:
\begin{eqnarray}\label{ipip1}
\left|\partial_\xi^\alpha(\SF\phi^m)(\xi)\right|\leq C_o
\end{eqnarray}
for $0\leq |\alpha|\leq N$\,($N\geq[\frac{n}{p}]+1$),  where $C_o$ is a constant independent on $m$.

Let
$\psi(x)\in S(\RR^n)$ to be fixed satisfying $\int\psi(x)dx\sim1$.
  $\varphi\in S(\RR^n)$ and $\varphi^k\in S(\RR^n)$  are defined as:
\begin{eqnarray*}
\left\{ \begin{array}{cc}
\varphi(\xi)=0 \ \ \hbox{for}\,|\xi|\geq1 \\
\\
\varphi(\xi)=1\ \ \hbox{for}\,|\xi|\leq1/2,
\end{array}\right.
\left\{ \begin{array}{cc}
\varphi^k(\xi)=\varphi(\xi) \ \ \hbox{for}\,k=0, \\
\\
\varphi^k(\xi)=\varphi(2^{-k}\xi)-\varphi(2^{1-k}\xi)\ \ \hbox{for}\,k\geq1.
\end{array}\right.
\end{eqnarray*}
Thus
$$1=\sum_{k=0}^{\infty}\varphi^k(\xi).$$
 Let $\eta_m^{k}$ to be defined as
$$(\SF\eta_m^{k})(\xi)=\frac{\varphi^k(\xi)(\SF\psi)(\xi)}{(\SF\phi^m)(2^{-k-k_0}\xi)},$$
where $\SF$ denotes the Fourier transform.
Then we could obtain that:
\begin{eqnarray*}
(\SF\psi)(\xi)&=&\sum_{k=0}^{\infty}\frac{\varphi^k(\xi)(\SF\psi)(\xi)}{(\SF\phi^m)(2^{-k-k_0}\xi)}(\SF\phi^m)(2^{-k-k_0}\xi)\\ \nonumber &=&\sum_{k=0}^{\infty}(\SF\eta_m^{k})(\xi)(\SF\phi^m)(2^{-k-k_0}\xi).
\end{eqnarray*}
Thus
\begin{eqnarray}\label{tan17y}
\psi(x)=\sum_{k=0}^{\infty}\eta_m^{k}\ast\phi^m_{2^{-k-k_0}}(x).
\end{eqnarray}
Notice   that the following hold
\begin{eqnarray}\label{tan14y}
\sup_{\xi\in\RR^n}\left||\xi|^{\alpha}\partial_{\xi}^{\alpha'}(\varphi^k(\xi)(\SF\psi)(\xi))\right|\lesssim_{\alpha, \alpha', M} 2^{-kM}\ \hbox{for\ any\ }M>0,\,\, 0\leq |\alpha|,
\end{eqnarray} where $C_{ \alpha'}$ is a
constant independent on $m$, together with Formula\,(\ref{ipip1}),  we could deduce that
\begin{eqnarray}\label{tan15y}
\sup_{\xi\in\RR^n}\left||\xi|^{\alpha}\partial_{\xi}^{\alpha'}((\SF\eta_m^{k})(\xi))\right|\leq C_{\alpha, \alpha', k_0, N}2^{-kN}\ \hbox{for}\,N\geq[\frac{n}{p}]+1, \ \ 0\leq |\alpha|\leq N
\end{eqnarray}
where $C_{\alpha, \alpha', k_0, N}$ is a
constant independent on $m$ and $k$. Then
\begin{eqnarray}\label{tan16y}
\left|\int_{\RR^n}\eta_m^{k}\left(u\right)\left(1+2^{k+k_0}|u|\right)^{N}du \right| \leq C_{k_0, N}2^{-k},
\end{eqnarray}
where $C_{k_0, N}$ is a constant independent on $m$.
 For $f\in  L^1(\RR^n)$, from Formulas\,(\ref{tan17y}), we could have
\begin{eqnarray}\label{tan18y}
M_{\psi}f(x)&=& \sup_{r>0}\left|\int_{\RR^n} f(y)\frac{1}{r^n}\psi\left(\frac{x-y}{r} \right)dy\right|
\\&=&\nonumber C\sup_{r>0}\sum_{k=0}^{+\infty}\left|\int_{\RR^n}\int_{\RR^n} f(y)\eta_m^{k}\left(\frac{s}{r}\right)\frac{1}{r^n}\phi^m\left(\frac{x-y-s}{2^{-k-k_0}r}\right)\frac{ds}{(2^{-k-k_0}r)^{n}}dy\right|
\\&\leq&\nonumber C \sum_{k=0}^{+\infty} \left|\int_{\RR^n}\eta_m^{k}\left(\frac{s}{r}\right)\left(1+\frac{|s|}{2^{-k-k_0}r} \right)^{N}\frac{ds}{r^n}\right| \sup_{ r>0, s\in\RR^n}\left|\int_{\RR^n} f(y)\phi^m\left(\frac{x-y-s}{r} \right)\left(1+\frac{|s|}{r} \right)^{-N}\frac{dy}{r^n}\right|,
\end{eqnarray}
where $C$ is a constant independent on $m$. From Formulas\,(\ref{tan16y}, \ref{tan18y}) we could obtain:
\begin{eqnarray}\label{tan27y}
M_{\psi}f(x)&\lesssim&  \sup_{ r>0, s\in\RR^n}\left|\int_{\RR^n} f(y)\phi^m\left(\frac{x-y-s}{r} \right)\left(1+\frac{|s|}{r} \right)^{-N}\frac{dy}{r^n}\right|
\\&\lesssim&\nonumber   \left(\sup_{0\leq |s|< r}+\sum_{k=1}^{\infty}\sup_{ 2^{k-1}r\leq |s|< 2^{k}r}\right)\left|\int_{\RR^n} f(y)\phi^m\left(\frac{x-y-s}{r} \right)\left(1+\frac{|s|}{r} \right)^{-N}\frac{dy}{r^n}\right|
\\&\lesssim&\nonumber   \sum_{k=0}^{+\infty}2^{-(k-1)N}\sup_{ 0\leq |s|< 2^{k}r}\left|\int_{\RR^n} f(y)\phi^m\left(\frac{x-y-s}{r} \right)\frac{dy}{r^n}\right|.
\end{eqnarray}
Formula\,(\ref{no1}) leads to
\begin{eqnarray}\label{tan19y}
& &\int_{\RR^n}\sup_{0\leq |s|< 2^{k}r}\left|\int_{\RR^n}f(y)\phi^m\left(\frac{x-y-s}{r}\right)\frac{dy}{r^n}\right|^pdx\\ \nonumber&\leq& C\left(1+2^k\right)^n\int_{\RR^n}\sup_{0\leq |s|<r}\left|\int_{\RR^n}f(y)\phi^m\left(\frac{x-y-s}{r}\right)\frac{dy}{r^n}\right|^pdx .
\end{eqnarray}
Notice that $N\geq[\frac{n}{p}]+1$, thus  Formulas\,(\ref{tan27y}, \ref{tan19y}) lead to
\begin{eqnarray}\label{tan20y}
 \int_{\RR^n} |M_{\psi}f(x)|^pdx\leq C \int_{\RR^n}\sup_{0\leq |s|<r}\left|\int_{\RR^n}f(y)\phi^m\left(\frac{x-y-s}{r}\right)\frac{dy}{r^n}\right|^pdx ,
 \end{eqnarray}
where $C$ is a constant independent on $m$. We use $\left(F^m f\right)(x, r)$ and $\left(F f\right)(x, r)$ to denote as:
\begin{eqnarray*}
 \left(F^m f\right)(x, r)=\int_{\RR^n}f(y)\phi^m\left(\frac{x-y}{r}\right)\frac{dy}{r^n},\ \ \ \  \left(F f\right)(x, r)=\int_{\RR^n}f(y)\phi\left(\frac{x-y}{r}\right)\frac{dy}{r^n}.
\end{eqnarray*}

Thus by Proposition\,\ref{exx6}(v), we could deduce the following inequality:
\begin{eqnarray}
\left|\left(F^m f\right)(u, r)-\left(Ff\right)(u, r)\right|&\leq& \int_{\RR^n} \left|f(y)\right|\left|\phi^m\left(\frac{u-y}{r}\right)-\phi\left(\frac{u-y}{r}\right)\right|\frac{dy}{r^n}
\\&\leq&\nonumber C \int_\RR \left|f(y)\right|\left(\frac{1}{m}\right)^{\gamma/2} \Big(1+\frac{|u-y|}{r}\Big)^{-n-\frac{\gamma}{2}}\frac{dy}{r^n}
\\&\leq&\nonumber C\sum_{k=0}^{+\infty} \Big(2^k\Big)^{-n-\frac{\gamma}{2}} 2^{nk} |M f(u)|\left(\frac{1}{m}\right)^{\gamma/2}
\\&\leq&\nonumber C |M f(u)|\left(\frac{1}{m}\right)^{\gamma/2}
\end{eqnarray}
where C is dependent on $\gamma$ and   $M$ is the Hardy-Littlewood Maximal Operator.
Let us set  $\delta_m(u)$ as: $$\delta_m(u)=\left|\left(F^m f\right)(u, r)-\left(Ff\right)(u, r)\right|.$$
Thus we could deduce the following:
$$\left\{x: \delta_m(x)>\alpha\right\}\subseteq \left\{x: M f(x) > \frac{1}{c} m^{\gamma/2}\alpha \right\}\ \ \ \hbox{for\,some\,constant\,c}.$$
Notice that $M$ is weak-(1, 1) bounded, thus the following holds for any $\alpha>0$:
$$\left|\left\{x: \delta_m(x)>\alpha\right\}\right|\leq \frac{c}{\alpha} \|f\|_{L^1(\RR^n)}\left(\frac{1}{m}\right)^{\gamma/2}.$$
Thus we could obtain:
$$\lim_{m\rightarrow+\infty}\left|\left\{x: \delta_m(x)>\alpha\right\}\right|=0.$$
Thus there exists a sequence $\{m_j\}\subseteq \{m\}$ such that the following holds: $$\lim_{m_j\rightarrow+\infty}\left(F^{m_j}f\right)(u, r)=\left(Ff\right)(u, r),\ \ \ a.e.u\in\RR^n $$ for $f\in L^1(\RR^n).$
Let us set $E$ as: $$E=\{u\in\RR^n: \lim_{m_j\rightarrow+\infty}\left(F^{m_j}f\right)(u, r)=\left(Ff\right)(u, r)\}.$$
Thus it is clear that  E is  dense in $\RR^n$. For any $x_0\in\RR^n$, there exists a $(u_0, r_0)$ with $r_0>0$, $u_0\in\RR^n$, $|u_0- x_0|<r_0$ such that the following holds:
$$\left|\left(F^{m_j}f\right)(u_0, r_0)\right| \geq \frac{1}{2} \sup_{|x_0-u|<r}|\left(F^{m_j} f\right)(u, r)|.$$
Notice  that $\left(F^{m_j}f\right)(u, r_0)$ is a continuous function in $u$ variable and E is  dense in $\RR^n$, thus there exists a $\widetilde{u}_0\in E$ with $|\widetilde{u}_0- x_0|<r_0$ such that
$$\left|\left(F^{m_j}f\right)(\widetilde{u}_0, r_0)\right|\geq \frac{1}{2}\left|\left(F^{m_j}f\right)(u_0, r_0)\right| \geq \frac{1}{4} \sup_{|x_0-u|<r}|\left(F^{m_j} f\right)(u, r)|.$$
Thus we could deduce that for any $x_0\in\RR^n$
\begin{eqnarray}\label{tan23y}
\sup_{\{u\in E: |x_0-u|<r\}} \left|\left(F^{m_j}f\right)(u, r)\right|\sim \sup_{\{u\in \RR^n: |x_0-u|<r\}} \left|\left(F^{m_j}f\right)(u, r)\right|.
\end{eqnarray}
Let the $\{Q_k\}_k$ to be the cubes with unit length satisfying $Q_i\bigcap Q_j=\emptyset$ when $i\neq j$, and $\bigcup_kQ_k=\RR^n$. From
Formula\,(\ref{tan23y}) together with the dominated convergence theorem, we could conclude:
\begin{eqnarray}\label{tan24y}
\overline{\lim}_{m_j\rightarrow+\infty}\int_{\RR^n} \sup_{|x-u|<r}\left|\left(F^{m_j} f\right)(u, r)\right|^pdx &\sim&\nonumber \overline{\lim}_{m_j\rightarrow+\infty}\int_{\RR^n} \sup_{\{u\in E: |u-x|<r\}} \left|\left(F^{m_j}f\right)(u, r)\right|^pdx
\\&\leq&\nonumber C \int_{\RR^n} \overline{\lim}_{m_j\rightarrow+\infty} \sup_{\{u\in E: |u-x|<r\}} \left|\left(F^{m_j}f\right)(u, r)\right|^pdx
\\&\leq&\nonumber C  \int_{\RR^n}\sup_{\{u\in E: |u-x|<r\}} \left|\left(Ff\right)(u, r)\right|^pdx
\\&\leq& C  \int_{\RR^n}\sup_{\{u\in \RR^n: |u-x|<r\}} \left|\left(Ff\right)(u, r)\right|^pdx.
\end{eqnarray}
By Formulas\,(\ref{tan20y}, \ref{tan24y}), we could deduce that the following inequality holds for $f\in L^1(\RR^n)$
\begin{eqnarray}\label{tan25yy}
\|f\|_{H^p(\RR^n)}\lesssim_{p, n, \gamma, \phi}\|(f*\phi)_{\nabla}\|_{L^p(\RR^n)} .
\end{eqnarray}
Notice that $H^p(\RR^n)\bigcap L^1(\RR^n)$ is dense in $H^p(\RR^n)$. Then   we could deduce that the Formula\,(\ref{tan25yy}) holds for  $f\in H^p(\RR^n)$. This proves our proposition.
\end{proof}

Thus from Proposition\,\ref{no11}, Proposition\,\ref{up0} and Proposition\,\ref{lp51}, we could obtain the following theorem:
\begin{theorem}\label{upoo}
$\phi(x)$ is a Lipschitz function $(\phi(x)\in \Lambda^{\gamma},\,\int\phi(x)dx\sim1)$,\,$(0<\gamma\leq1), \phi(x)>0$ without compact support in $\RR^n$ satisfying Formulas\,(\ref{up1}, \ref{up2}), then we could deduce the following inequality for $\frac{1}{1+\gamma}<p\leq1$,  $f\in H^p(\RR^n)$:
\begin{eqnarray}\label{tan28y}
\|f\|_{H^p(\RR^n)}\sim_{p, \gamma, \phi}\|(f*\phi)_{\nabla}\|_{L^p(\RR^n)}\sim_{ p,  \gamma, \phi} \|(f*\phi)_{+}\|_{L^p(\RR^n)}.
\end{eqnarray}
\end{theorem}

\textbf{Declarations}:

\textbf{Availability of data and materials}:Not applicable

\textbf{Competing interests}: The authors declare that they have no competing interests

\textbf{Funding}:Not applicable

\textbf{Authors’ contributions}:The author ZhuoRan Hu finished this paper alone.

\textbf{Acknowledgements}:Not applicable

\textbf{Author}: ZhuoRan Hu

\end{document}